\documentclass[12pt]{amsart}

\usepackage{amssymb,latexsym,amsmath,amscd,mathrsfs,amsthm}
\usepackage{verbatim,graphics,xypic}
\usepackage{fullpage}

\theoremstyle{plain}
\newtheorem{theorem}{Theorem}[section]
\newtheorem{lemma}[theorem]{Lemma}
\newtheorem{proposition}[theorem]{Proposition}

\newtheorem{corollary}[theorem]{Corollary}
\theoremstyle{definition}
\newtheorem{definition}[theorem]{Definition}

\theoremstyle{remark}
\newtheorem{remark}[theorem]{Remark}

\newtheorem*{notation*}{Notation}

\newcommand{\Space}[1]{\ensuremath{\mathfrak{#1}}}
\newcommand{\field}[1]{\ensuremath{\mathbb{#1}}}
\newcommand{\geom}[1]{\ensuremath{\mathbb{#1}}}
\newcommand{\inner}[2]{\left\langle #1,#2 \right\rangle}

\newcommand{\inv}{^{-1}}
\newcommand{\D}{\mathrm{d}}
\newcommand{\id}{\mathrm{id}}

\newcommand{\gen}[1]{{\langle #1 \rangle}}

\newcommand{\R}{\field{R}}
\newcommand{\C}{\field{C}}

\newcommand{\Z}{\field{Z}}

\renewcommand{\phi}{\varphi}

\newcommand{\eps}{\varepsilon}

\newcommand{\Lie}[1]{\mathrm{#1}}

\usepackage{graphicx}
\newcommand{\N}{\mathbb{N}}
\newcommand{\dfn}[1]{\emph{#1}}

\begin{document}
\title{The complete family of Arnoux--Yoccoz surfaces}
\author{Joshua P. Bowman}
\address{Institute of Mathematical Sciences, 
Stony Brook University, Stony Brook, NY 11794}
\date{\today}

\begin{abstract}
The family of translation surfaces $(X_g,\omega_g)$ constructed 
by Arnoux and Yoccoz from self-similar interval exchange maps 
encompasses one example from each genus $g$ greater than or 
equal to $3$. We triangulate these surfaces and deduce general 
properties they share. The surfaces $(X_g,\omega_g)$ converge 
to a surface $(X_\infty,\omega_\infty)$ of infinite genus and 
finite area. We study the exchange on infinitely many intervals 
that arises from the vertical flow on $(X_\infty,\omega_\infty)$ 
and compute the affine group of $(X_\infty,\omega_\infty)$, 
which has an index $2$ cyclic subgroup generated by a hyperbolic 
element. 
\end{abstract}

\maketitle

\setcounter{tocdepth}{1}
\tableofcontents

\section{Introduction}\label{S:intro}

\subsection{From the golden ratio to the geometric series}

From our calculus courses, we know that the infinite geometric 
series $\frac{1}{2} + \frac{1}{4} + \frac{1}{8} + \cdots$ 
converges to $1$. Indeed, using the summation formula 
$\sum_{k=1}^\infty x^k = x/(1-x)$, we find that $\frac{1}{2}$ is 
the unique solution to the equation $\sum_{k=1}^\infty x^k = 1$. 
From even earlier in our lives, perhaps, we recall that the 
equation $x + x^2 = 1$ has a unique positive solution, whose 
inverse is the golden ratio. The expression $x + x^2$ may be 
viewed as a partial geometric series, which can be extended to $n$ 
terms: $x + \cdots + x^n$.

The positive solutions to the equations $x + \cdots + x^n = 1$ for 
$n \ge 3$ are instrumental in creating a certain family of 
measured foliations on surfaces, which was introduced by 
P.~Arnoux and J.-C.~Yoccoz in 1981 \cite{pAjcY81}. 
In contemporary terminology, these measured foliations are the 
vertical foliations of certain translation surfaces. These 
surfaces were discovered in an attempt to provide examples of 
pseudo-Anosov homeomorphisms, which had been defined only a few 
years previously by W.~Thurston in his classification of surface 
homeomorphisms \cite{wT88}. It was shown some time later (2005) 
by P.~Hubert and E.~Lanneau \cite{pHeL05} that the Arnoux--Yoccoz 
examples do not arise from the Thurston--Veech construction via 
compositions of multi-twists \cite{wT88,wV89}; in particular, 
their affine groups contain no parabolic elements.

In this paper, after providing some background on translation 
surfaces, we will present the surfaces constructed by Arnoux and 
Yoccoz and give explicit triangulations, then use these to prove 
certain properties common to all these surfaces. We will also see 
that the family can be extended to include the cases $n = 2$ and 
$n = \infty$. These extreme cases will turn out to be exceptional 
in their construction---the first corresponds to a singular 
surface (see the Appendix) and the second to a surface of 
infinite type (see \S\ref{S:limit})---but we hope that the 
self-similarity property that the golden ratio and the geometric 
series share with all of the other examples (see \S\ref{S:iettri}) 
will illuminate the entire sequence of surfaces for the reader.

\subsection{Background on translation surfaces}

There are two commonly accepted definitions for a ``translation 
surface'': either a surface with a translation atlas, or a 
Riemann surface with an abelian differential $\omega$. These 
definitions are not quite equivalent. The former endows the 
surface with a Riemannian metric (given in a translation chart 
$z$ by $|\D{z}|^2$) so that it is everywhere locally isometric 
to the Euclidean plane. The latter allows a discrete set of 
points on the surface to have neighborhoods isometric to ``cone 
points'' with respect to the metric $|\omega|^2$; these 
``singularities'' of the metric occur at zeroes of $\omega$ and 
have angles that are integer multiples of $2\pi$. This latter 
convention is necessary, for instance, in order to have compact 
translation surfaces of genus $\ge 2$. Yet it is not hard to move 
from the complex-analytic definition to the Riemannian definition 
by simply ``puncturing'' the surface at the cone points. For 
convenience, then, we adopt the following convention.

\begin{definition}
A {\em translation surface} is a pair $(X,\omega)$, where $X$ 
is a Riemann surface and $\omega \ne 0$ is a holomorphic 1-form 
on $X$. 
\end{definition}

Note that $X$ is not assumed to be compact in the above 
definition.

\begin{definition}
Let $(X,\omega)$ be a translation surface. A homeomorphism 
$\phi : X \to X$ is called {\em affine} if it is affine with 
respect to the canonical charts of $\omega$. The group of affine 
homeomorphisms from $X$ to itself is denoted 
$\Lie{Aff}(X,\omega)$.
\end{definition}

Any affine homeomorphism $\phi$ has a globally well-defined 
derivative $\mathrm{der}\,\phi \in \Lie{GL}_2(\R)$; this is 
essentially because the group of translations is normal in the 
group of all affine bijections of $\R^2$. If $(X,\omega)$ has 
finite area, then necessarily any $\phi \in \Lie{Aff}(X,\omega)$ 
satisfies $\det(\mathrm{der}\,\phi) = \pm 1$; in this case, the 
dynamical type of the map $\phi$ can be easily determined 
\cite{wT88,iK81,rC04}: let $\mathrm{Tr}$ denote the trace 
function.
\begin{itemize}
\item If $|\mathrm{Tr}\,\mathrm{der}\,\phi| < 2$, then $\phi$ 
has finite order.
\item If $|\mathrm{Tr}\,\mathrm{der}\,\phi| = 2$, then 
$(X,\omega)$ decomposes into parallel cylinders such that on 
each cylinder some power of $\phi$ acts as a power of a Dehn 
twist.
\item If $|\mathrm{Tr}\,\mathrm{der}\,\phi| > 2$, then $\phi$ 
is pseudo-Anosov.
\end{itemize}
The importance of the group $\{\mathrm{der}\,\phi \mid \phi \in 
\Lie{Aff}(X,\omega)\}$ for compact $X$ was first observed by 
Veech. We make the following general definition 
\cite{wV89,cEfG97}.

\begin{definition}
Let $(X,\omega)$ be a translation surface. The image of the 
derivative map $\mathrm{der} : \Lie{Aff}(X,\omega) \to 
\Lie{GL}_2(\R)$ is called the {\em Veech group} of $(X,\omega)$ 
and is denoted $\Gamma(X,\omega)$.
\end{definition}

A great deal of general theory about compact translation 
surfaces (and their moduli spaces) has been developed, by too 
many authors to name. Recently, several classes of non-compact 
translation surfaces and their Veech groups have been studied. 
Many of these are surfaces that cover compact surfaces and whose 
Veech groups are therefore contained in Veech groups of compact 
surfaces (e.g., \cite{pHpHbW08,pHsLsT09}). Notable exceptions are 
a surface made from two infinite polygons inscribed in parabolas, 
which can be obtained as a limit of Veech's original examples 
\cite{pH08}, and a family of ``hyperelliptic'' surfaces of finite 
area and infinite genus \cite{rC04}. The surface we present in 
\S\ref{S:limit} combines certain features of these last two 
examples: it is a geometric limit of compact surfaces, and it 
has finite area. Future work on other such limits seems called 
for; at the end of this work, we discuss a possible direction for 
research on Veech groups of geometric limits of translation 
surfaces (Remark~\ref{R:veech}).

\section{From intervals to triangles}\label{S:iettri}

\subsection{Pisot numbers and interval exchange maps}\label{S:iet}

In this section we review the algebraic numbers and interval 
exchange maps involved in the construction of the Arnoux--Yoccoz 
translation surfaces. Given any $g \ge 2$, the polynomial 
\begin{equation}\label{Eq:pisot}
x^g + x^{g-1} + \cdots + x - 1
\end{equation}
has a unique positive root, since its values at $0$ and $1$ are 
$-1$ and $g-1$, respectively, and its derivative is positive for 
all positive $x$. We denote the positive root of \eqref{Eq:pisot} 
simply as $\alpha$, suppressing its dependence on $g$. Arnoux and 
Yoccoz showed that the inverse of $\alpha$ is a Pisot number, 
which means that $\alpha$ is in fact the only root of 
\eqref{Eq:pisot} that lies within the unit disk. Hubert and 
Lanneau remarked that, if $g$ is even, then \eqref{Eq:pisot} has 
one negative root, and if $g$ is odd, then $\alpha$ is the only 
real root. We add to these properties the following:

\begin{lemma}\label{L:1/2}
For each $g \ge 2$, the positive root $\alpha$ of 
\eqref{Eq:pisot} satisfies 
\begin{equation}\label{Eq:expconv}
\frac{1}{2^{g+2}} < \alpha - \frac{1}{2} < \frac{1}{2^{g+1}}.
\end{equation}
\end{lemma}


\begin{proof}
To obtain the lower bound, we claim that, when 
$r = 1/2+1/2^{g+2}$, the polynomial \eqref{Eq:pisot} evaluated 
at $r$ is negative. This is equivalent to 
\[
\frac{1 - r^{g+1}}{1 - r} < 2, 
\qquad\text{or}\qquad 
\left( 1 + \frac{1}{2^{g+1}} \right)^{g+1} > 1,
\]
which is true for all $g \ge 2$. The upper bound is obtained 
similarly.
%
\end{proof}

\begin{figure}
\centering
\includegraphics[scale=.9]{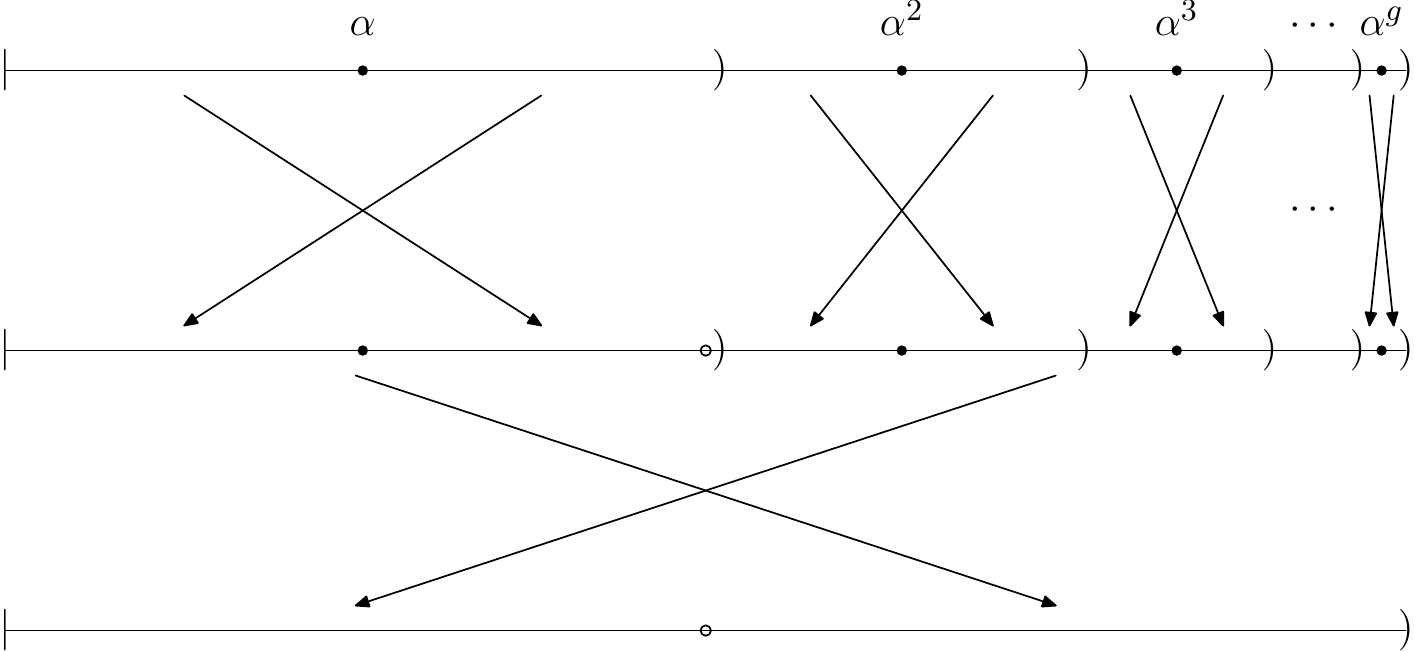}
\caption{The interval exchange $f_g$ as a composition of two 
involutions.}\label{F:ayiet}
\end{figure}

Arnoux and Yoccoz \cite{pAjcY81} introduced an interval exchange 
map based on the geometric properties of $\alpha$. First, the unit 
interval is subdivided into $g$ intervals of lengths $\alpha$, 
$\alpha^2$, \dots, $\alpha^g$. Each of these subintervals is 
divided in half, and the halves are exchanged within each 
subinterval. Finally the entire unit interval is divided into 
half, and these two halves are exchanged. We denote the total 
process $f_g$ (see Figure~\ref{F:ayiet}). We will occasionally be 
interested in the behavior of $f_g$ and its iterates on the 
endpoints of the subintervals, so for specificity we restrict the 
map to $[0,1)$ and assume that the left endpoint of each piece is 
carried along. The key feature of $f_g$ is its self-similarity:

\begin{proposition}[Arnoux--Yoccoz]
Let $\tilde{f_g}$ be the interval exchange map induced on 
$[0,\alpha)$ by the first return map of $f_g$. Then $f_g$ is 
conjugate to $\tilde{f_g}$.
\end{proposition}

The proof uses an explicit piecewise affine map 
$h_g : [0,1) \to [0,\alpha)$, defined as follows:
\[
h_g(x) = \begin{cases}
\alpha x + \frac{\alpha+\alpha^{g+1}}{2}, 
	& x \in \left[0,\frac{1-\alpha^g}{2}\right) \\
\alpha x - \frac{\alpha-\alpha^{g+1}}{2}, 
	& x \in \left[\frac{1-\alpha^g}{2},1\right)
\end{cases}
\]
which satisfies $f_g = h_g\inv \circ \tilde{f_g} \circ h_g$. In 
\S\ref{S:limit}, we will show similar kinds of results for certain 
exchanges on infinitely many subintervals. 

In their original paper, citing work of G.~Levitt, Arnoux and 
Yoccoz state that, for a given interval exchange map:
\begin{quotation}
{\em \dots\ on peut construire une suspension canonique, et l'on 
sait que toute suspension poss\'edant les m\^emes singularit\'es 
(en type et en nombre) que cette suspension canonique lui est 
hom\'eo\-morphe par un hom\'eo\-morph\-isme pr\'eservant la mesure 
transverse du feuilletage.}
\end{quotation}
(The ``canonical suspension'' is a measured foliation on a 
compact surface together with a closed curve transverse to 
the foliation on which the first return map of the foliation 
induces the given interval exchange map.) 
They then use this result and the self-similarity of $f_g$ to 
demonstrate the existence of a pseudo-Anosov homeomorphism 
$\psi_g$ on a surface of genus $g$ such that the expansion 
constant of $\psi_g$ is $1/\alpha$. In a separate paper 
\cite{pA88}, Arnoux gives an explicit description of the canonical 
suspension of $f_3$ and illustrates $\psi_3$. In 
\S\S\ref{S:step2tri}--\ref{S:tri} we will present the 
generalization of Arnoux's construction to all genera and 
exploit these presentations to make further conclusions about 
the Arnoux--Yoccoz surfaces.

\subsection{Steps and slits}\label{S:step2tri}

Fix $g \ge 3$. In this section, we will present the genus $g$ 
Arnoux--Yoccoz surface $(X_g,\omega_g)$ by generalizing Arnoux's 
presentation of $(X_3,\omega_3)$. Starting with a unit square, we 
carve out a ``staircase'' in the upper right-hand corner, with the 
widths of the steps, from left to right, given by $\alpha$, 
$\alpha^2$, \dots, $\alpha^g$, and the distances between the 
steps, going down, given by $\alpha^g$, $\alpha^{g-1}$, \dots, 
$\alpha$. We further slit this square along several vertical 
segments $\sigma_1$, $\sigma_2$, \dots, $\sigma_g$. The slits are 
made starting along the bottom edge of the square at points whose 
$x$-coordinates are images by $f_g$ of the left-hand endpoints of 
the intervals $\big[\frac{\alpha-\alpha^i}{1-\alpha},
\frac{\alpha-\alpha^{i+1}}{1-\alpha}\big)$, for $1 \le i \le g$. 
({\sc n.b.}: the lower endpoints of the slits are not 
singularities on the resulting surface, following the rest of 
the construction below.) 

Now we wish to provide appropriate gluings for the surface to have 
an affine self-map. These identifications are as follows: 
\begin{itemize}
\item The tops of the steps are glued to the bottom of the unit 
square according to the interval exchange $f_g$.
\item The vertical edge of the bottommost step, having length 
$\alpha$, is identified with the bottom portion of the leftmost 
vertical edge. 
\item The remaining top portion of the leftmost edge of the 
square, having length $1-\alpha$, is identified with the bottom 
portion to the left of $\sigma_1$.
\item The remaining top portion to the left of $\sigma_1$ is 
identified with the right side of $\sigma_g$.
\item The vertical edge of the step having height $\alpha^i$ 
($2 \le i \le g$) is identified with the bottom portion to the 
right of the segment $\sigma_{i-1}$. 
\item The remaining top portion to the right of each segment 
$\sigma_i$ ($1 \le i \le g-1$) is identified with the left side 
of the segment $\sigma_{i+1}$.
\end{itemize}

\begin{figure}[h]
\centering
\includegraphics{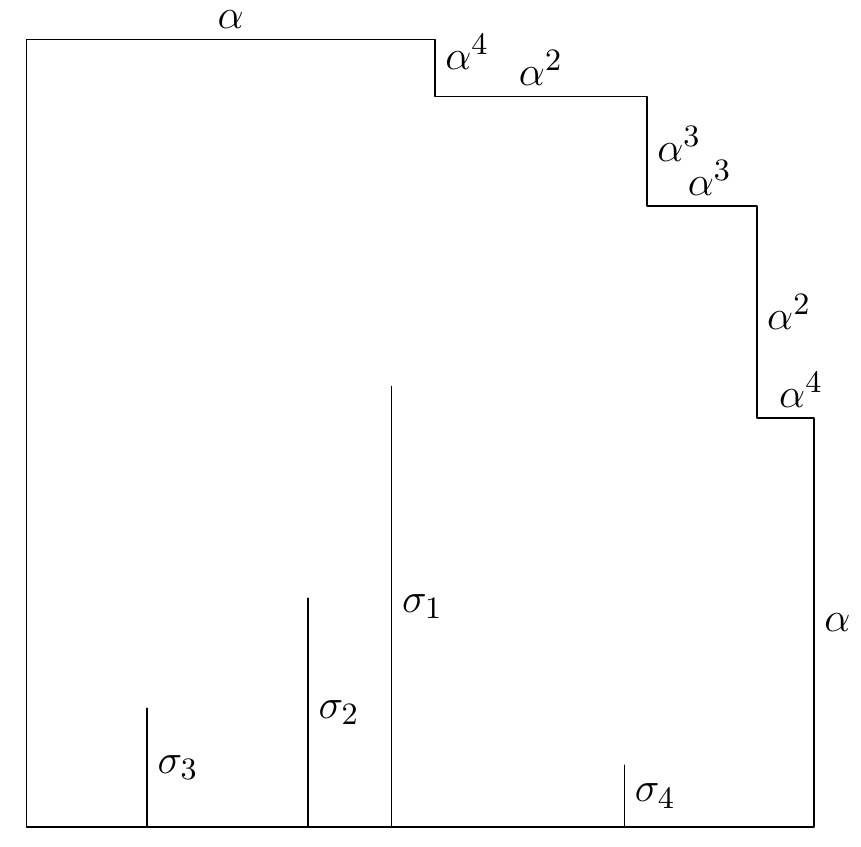}
\caption{The steps and slits for the genus $4$ Arnoux--Yoccoz 
surface.}
\end{figure}

There is a one real-parameter family of surfaces that satisfy the 
gluings given above; the easiest parameter to vary is 
$|\sigma_g|$. We want to single out a value for this parameter so 
that the surface admits a pseudo-Anosov affine map. The required 
condition is described by the equation 
$\alpha (1 + |\sigma_g|) = (1 - \alpha) + |\sigma_g|$,
which says that the length of $\sigma_1$ is $\alpha$ times the sum 
of the length of $\sigma_g$ and the length of the left edge of the 
square (i.e., $1$). Solving this equation, we find 
$|\sigma_g| = (2\alpha - 1)/(1 - \alpha)$, which determines the 
lengths of the remaining slits.

The pseudo-Anosov homeomorphism $\psi_g : X_g \to X_g$ expands the 
horizontal foliation of $\omega_g$ by a factor of $1/\alpha$ and 
contracts the vertical foliation by a factor of $\alpha$. It 
permutes the vertical segments in a predictable manner: for each 
$i$ from $1$ to $g-1$, $\psi_g$ sends $\sigma_i$ to 
$\sigma_{i+1}$, and also sends the union of $\sigma_g$ with the 
left-hand edge of the initial square to $\sigma_1$. The step of 
height $\alpha^i$ is also sent to the step of height 
$\alpha^{i+1}$ ($1 \le i \le g-1$).

\subsection{Triangulation of $(X_g,\omega_g)$}\label{S:tri}

Let $g$ be as in \S\ref{S:step2tri}. Now we give an alternate 
construction of the surface $(X_g,\omega_g)$ from $4g$ triangles. 
Begin with the points $P_0,\dots,P_g, Q_0,\dots,Q_g$ in $\R^2$, 
chosen as follows (see Figure~\ref{F:pts}): 
\begin{gather*}
P_0 = \left(\frac{1-\alpha^g}{2},\frac{\alpha^2}{1-\alpha}\right),
\qquad
Q_0 = \left(-\frac{\alpha^g}{2},\alpha\right), \\
P_1 = \left(
         -\frac{\alpha^{g-1}+\alpha^g}{2},
         \frac{\alpha-\alpha^2+\alpha^3}{1-\alpha}
         \right),
\\
P_g = \left(
		1+\frac{\alpha-\alpha^g}{2},
		\frac{3\alpha-1-\alpha^2}{1-\alpha}
		\right), \\
P_i = \left(
		\frac{\alpha-\alpha^i}{1-\alpha},
		\frac{\alpha}{1-\alpha}
		\right) 
\qquad\text{for $i = 2, \dots, g-1$}, \\
Q_i = \left(
		\frac{2\alpha-\alpha^i-\alpha^{i+1}}{2(1-\alpha)},
		\frac{\alpha-\alpha^{g-i+2}}{1-\alpha}
		\right)
\qquad\text{for $i = 1, \dots, g$}.
\end{gather*}

For $i = 1, \dots, g$, set $T_i = P_0 Q_i Q_{i-1}$ and 
$T_{g+i} = P_i Q_{i-1} Q_i$. For $i = 1,\dots, 2g$, let $T'_i$ be 
the reflection of $T_i$ in the horizontal axis. Glue the $T_i$s 
along their common boundaries, and likewise for the $T_i'$s. Then 
each remaining ``free'' edge is a translation of another; we glue 
each such pair of edges:
\begin{itemize}
\item $P_0 Q_0$ is paired with $P'_0 Q'_g$, and $P'_0 Q'_0$ is 
paired with $P_0 Q_g$.
\item $P_1 Q_1$ is paired with $P'_g Q'_{g-1}$, and $P'_1 Q'_1$ is 
paired with $P_g Q_{g-1}$.
\item $P_1 Q_0$ is paired with $P_{g-1} Q_{g-1}$, and $P'_1 Q'_0$ 
is paired with $P'_{g-1} Q'_{g-1}$.
\item $P_g Q_g$ is paired with $Q_1 P_2$, and $P'_g Q'_g$ is 
paired with $Q'_1 P'_2$.
\item For $i = 2, \dots, g-2$, $P_i Q_i$ is paired with 
$Q'_i P'_{i+1}$ and $P'_i Q'_i$ is paired with $Q_i P_{i+1}$.
\end{itemize}
All of the $P_i$s and $Q'_i$s are identified to become a cone 
point, and likewise for all of the $Q_i$s and $P'_i$s. The 
resulting surface therefore has genus $g$ and lies in the 
stratum $\mathcal{H}(g-1,g-1)$. 


\begin{figure}
\centering
\includegraphics{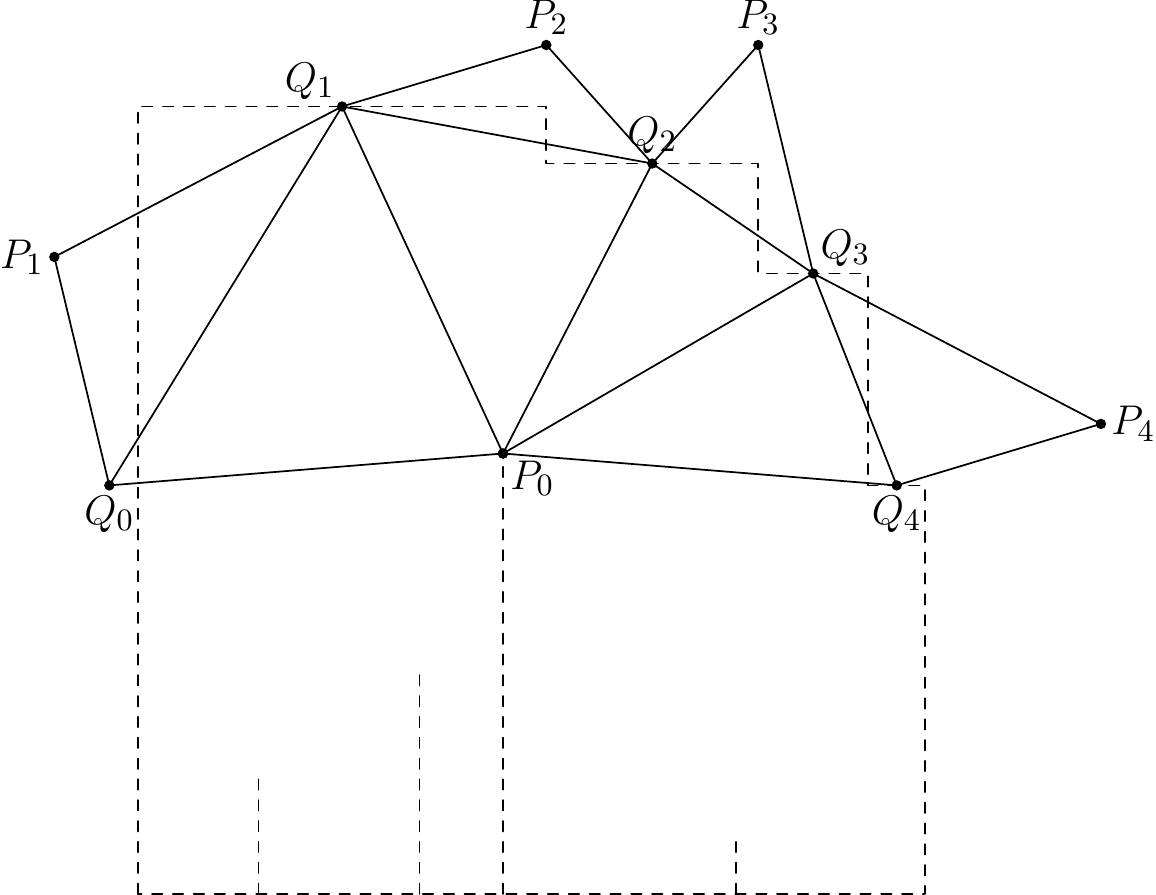}
\caption{The points $P_0,\dots,P_4,Q_0,\dots,Q_4$ relative to 
$(X_4,\omega_4)$'s staircase.}\label{F:pts}
\end{figure}

\begin{figure}[h]
\centering
\includegraphics{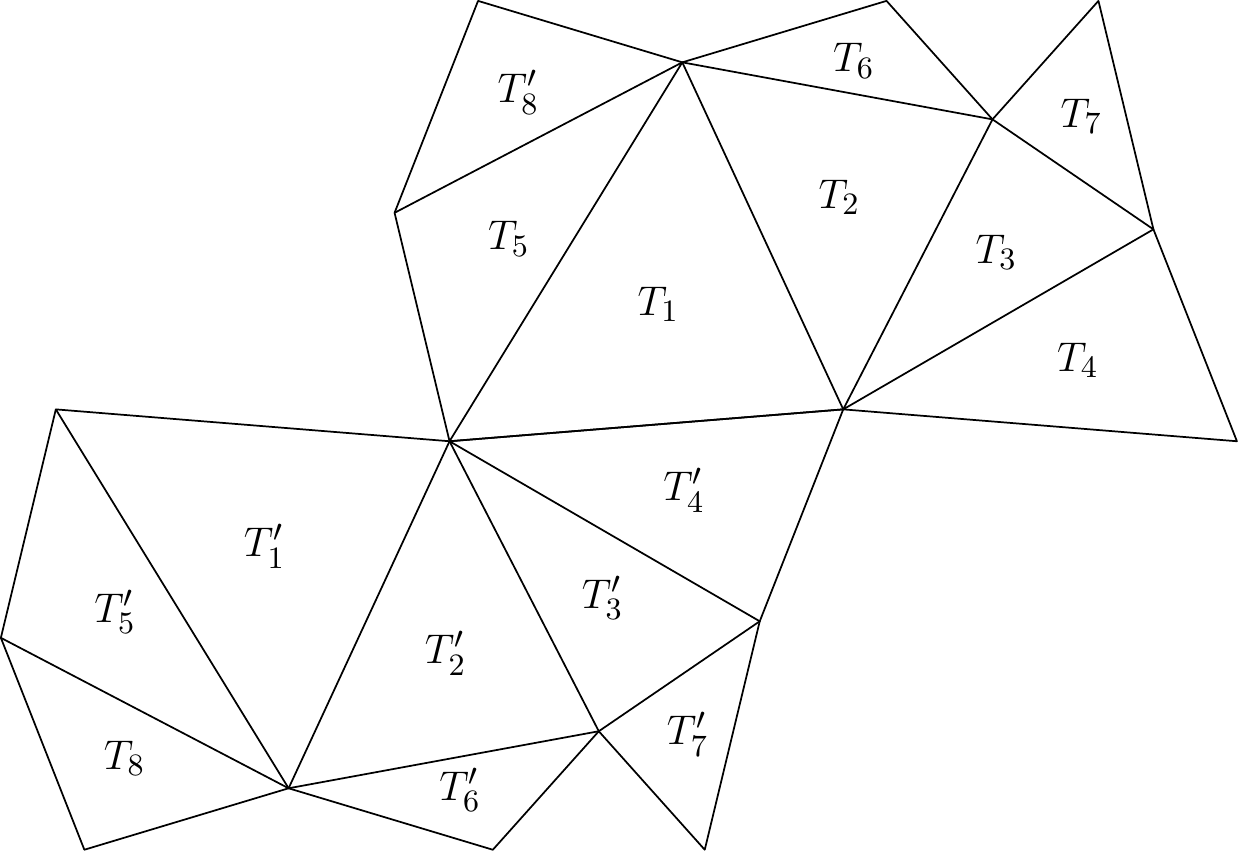}
\caption{A triangulation of $(X_4,\omega_4)$.}
\label{F:tri}
\end{figure}

One can verify the following result directly by checking that the 
surface we have constructed from triangles is isometric to the 
staircase presentation (cf.\ Figures \ref{F:pts} and \ref{F:tri}).

\begin{proposition}
The $T_i$s and $T'_i$s induce a triangulation of $(X_g,\omega_g)$.
\end{proposition}

By ``triangulation'' we mean the structure of a $\Delta$-complex, 
in the sense of Hatcher \cite{aH02}; we also require that the 
set of vertices contain the cone points and the $1$-cells be 
geodesic.

\begin{corollary}
$\Lie{Aff}(X_g,\omega_g)$ contains a fixed-point free, 
orientation-reversing involution $\rho_g$, which commutes with 
$\psi_g$, and whose derivative is reflection in the $x$-axis.
\end{corollary}

The existence of this symmetry occurs for a completely general 
reason: $f_g$ is conjugate to its inverse by the following 
``rotation'' of the unit interval: 
\[
r(x) = \begin{cases}
x + \frac{1}{2}, & x \in [0,\frac{1}{2}) \\
x - \frac{1}{2}, & x \in [\frac{1}{2},1)
\end{cases}
\]
By the reasoning invoked in \S\ref{S:iet}, the surface obtained 
from $(X_g,\omega_g)$ by applying complex conjugation to the 
charts of $\omega_g$ (which is a suspension of $f_g\inv$, and 
therefore of $f_g$) is translation equivalent to $(X_g,\omega_g)$ 
itself, which yields the existence of $\rho_g$.

\begin{corollary}
The compact non-orientable surface of Euler characteristic $1 - g$ 
admits a pseudo-Anosov homeomorphism whose invariant foliations 
have one singular point and whose expansion constant has degree 
$g$.
\end{corollary}

This corollary generalizes a result from \cite{pAjcY81}, in which 
it is shown that $(X_3,\omega_3)$ can also be constructed by 
lifting a measured foliation on $\geom{RP}^2$ first to the 
non-orientable surface of Euler characteristic $-2$ and then 
to genus $3$.

\begin{corollary}\label{C:nothyp}
If $g \ge 4$, then $X_g$ is not hyperelliptic.
\end{corollary}

\begin{proof}
Every abelian differential on a hyperelliptic surface is odd with 
respect to the hyperelliptic involution. If, for some $g \ge 4$, 
$X_g$ were hyperelliptic, then there would have to be an isometry 
of $(X_g,\omega_g)$ with derivative $-\id$. Such an isometry 
would, for instance, have to send the segment $P_{g-1} Q_{g-1}$ 
to a parallel segment of the same length. This segment cannot be 
preserved by the isometry, because it would have to be rotated 
around its midpoint---but $Q_{g-2}$ (which is opposite 
$P_{g-1} Q_{g-1}$ in the triangle $T_{2g-1}$) has no potential 
image on the other side of $P_1 Q_0$ (which is identified with 
$P_{g-1} Q_{g-1}$). It is easily checked that no other saddle 
connections on $(X_g,\omega_g)$ are parallel to $P_{g-1} Q_{g-1}$ 
and have the same length. Hence no isometry with derivative 
$-\id$ exists.
\end{proof}

\begin{remark}
The surface $X_3$ is well-known to be hyperelliptic. (In 
\cite{jpb10}, Weierstrass equations are given for two 
surfaces affinely equivalent to $(X_3,\omega_3)$; see also 
\cite{pHeLmM09}.) The obstruction described in the proof of 
Corollary~\ref{C:nothyp} does not occur in genus $3$, because 
the segment $P_2 Q_2$ does in fact have another saddle 
connection with the same length and direction, namely, $P_0 Q_1$. 
The hyperelliptic involution of $X_3$ exchanges each $T_i$ and 
$T_{g+i}$ by rotating around the midpoint of their common edge.
\end{remark}

\section{A limit surface: $(X_\infty,\omega_\infty)$}
\label{S:limit}

Lemma~\ref{L:1/2} implies that each triangle that appears in the 
construction of some $(X_g,\omega_g)$ has a ``limiting position''; 
from these we can construct a ``limit surface'' of infinite genus. 
See Figure~\ref{F:infty} for the definition of this surface. 
To be precise, we obtain a non-compact translation surface 
$(X_\infty,\omega_\infty)$, where $X_\infty$ has infinite genus, 
whose metric completion is the one-point compactification of 
$X_\infty$. Here, as usual, $\omega_\infty$ is the $1$-form 
induced on the quotient by $\D{z}$ in the plane. In a sense, 
the two cone points of the $(X_g,\omega_g)$, $g < \infty$, have 
``collapsed'' into each other, leaving an essential singularity 
at which all of the ``curvature'' of the space 
$(X_\infty,\omega_\infty)$ is concentrated. We shall briefly 
address in \S\ref{S:ietaff} the behavior of 
$(X_\infty,\omega_\infty)$ near this singular point. A 
\dfn{critical trajectory} of $(X_\infty,\omega_\infty)$ is a 
geodesic trajectory that leaves every compact subset of 
$X_\infty$. A \dfn{saddle connection} of 
$(X_\infty,\omega_\infty)$ is a geodesic trajectory (of finite 
length) that leaves every compact subset of $X_\infty$ in both 
directions.

\begin{figure}[h]
\centering
\includegraphics[scale=.9]{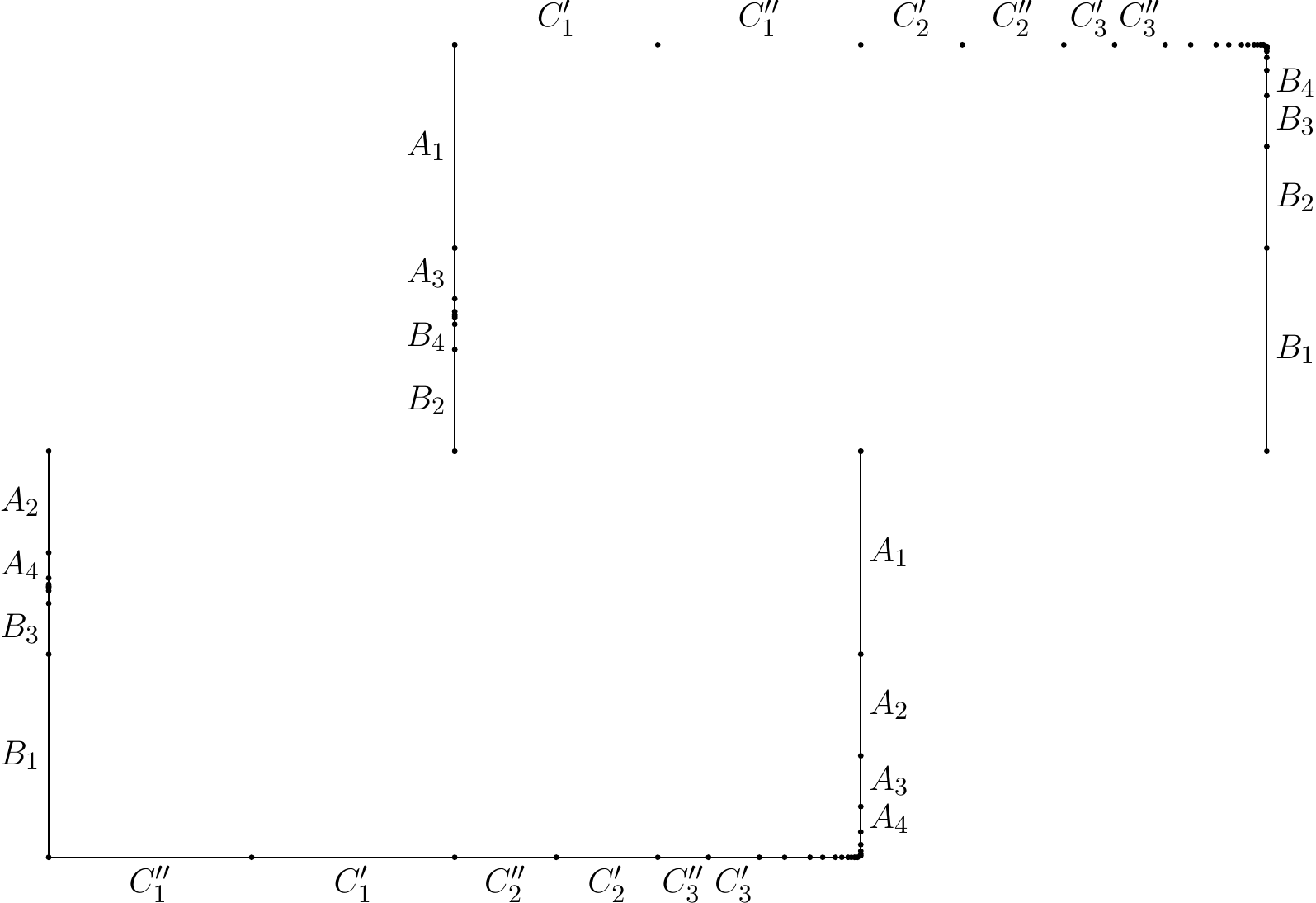}
\caption[The surface $(X_\infty,\omega_\infty)$.]%
{The surface $(X_\infty,\omega_\infty)$. Each pair of 
edges with the same label is identified by translation, as 
is the remaining pair of unlabeled edges. The length of each 
$A_n$, $B_n$, $C'_n$, or $C''_n$ is $1/2^{n+1}$.}\label{F:infty}
\end{figure}

\begin{theorem}\label{T:infty}
$X_\infty$ is a Riemann surface of infinite genus with one end, 
and $\omega_\infty$ is an abelian differential of finite area on 
$X_\infty$ without zeroes on $X_\infty$. 
$\Lie{Aff}(X_\infty,\omega_\infty)$ includes an 
orientation-reversing isometric involution $\rho_\infty$ without 
fixed points on $X_\infty$ and a pseudo-Anosov homeomorphism 
$\psi_\infty$ with expansion constant $2$. These two elements 
commute.
\end{theorem}
\begin{proof}
(In this paragraph, we follow the method of proof used by 
R.~Chamanara in \cite{rC04}.) That $X_\infty$ is a Riemann surface 
is evident, as are the claims about $\omega_\infty$. The fact that 
$X_\infty$ has infinite genus can be deduced from the existence of 
a set of pairwise non-homotopic simple closed curves 
$\{\gamma'_n,\gamma''_n\}_{n\in\N}$, where $\gamma'_n$ 
(respectively, $\gamma''_n$) connects the midpoints of the edges 
labelled $C'_n$ (respectively, $C''_n$), and each $\gamma'_n$ 
intersects only $\gamma''_n$ (and vice versa). To show that 
$X_\infty$ has only one topological end, we construct a sequence 
of compact subsurfaces with boundary. Let $K_g$ be the complement 
of the union of the open squares having side length $1/2^{g+1}$ 
and centered at the endpoints of the segments $A_n$, $B_n$, 
$C_n'$, $C_n''$. These $K_g$ satisfy $K_g \subset K_{g+1}$ and 
$\bigcup K_g = X_\infty$, and the complement of each $K_g$ has 
one component. Therefore by definition $X_\infty$ has one 
topological end.

The orientation-reversing affine map $\rho_\infty$ is visible in 
Figure~\ref{F:infty} as a glide-reflection in a horizontal axis 
with translation length $1/2$. It sends the interior of the upper 
rectangle to the interior of the lower rectangle, each edge 
labeled $A_n$ to an edge labeled $B_n$, and each edge labeled 
$C'_n$ to an edge labeled $C''_n$. Therefore it has no fixed 
points.

Now we demonstrate the pseudo-Anosov affine map $\psi_\infty$. Let 
$R$ be the central rectangle in Figure~\ref{F:infty}, and let 
$S_1$ and $S_2$ be the squares in the lower left and upper right, 
respectively. Expand $R$ horizontally by a factor of $2$, and 
contract $R$ vertically by a factor of $1/2$ to obtain 
$\psi_\infty(R)$. Do the same with the rectangle $R'$ which is the 
union of $S_1$ and $S_2$ (the top edge of $S_1$ is glued to the 
bottom edge of $S_2$) to obtain $\psi_\infty(R')$. Take 
$\psi_\infty(R)$ and lay it over $S_1$ and the lower half of $R$, 
and lay $\psi_\infty(R')$ over $S_2$ and the top half of $R$. This 
affine map is compatible with all identifications. That 
$\psi_\infty$ and $\rho_\infty$ commute may be checked directly.
\end{proof}

\begin{remark}
The pseudo-Anosov map $\psi_\infty : X_\infty \to X_\infty$ is a 
variant of the well-studied baker map, and thus 
$(X_\infty,\omega_\infty)$ is an alternate infinite-genus 
realization of this map, which was demonstrated on a 
``hyperelliptic'' infinite-genus surface by 
Chamanara--Gardiner--Lakic \cite{rCfGnL06}. The topological type 
of $X_\infty$ is that of a ``Loch Ness monster'' and is therefore 
related to the surfaces described in \cite{pPgSfV09}, although the 
flat structure of $\omega_\infty$ does not fall into the class of 
surfaces studied there.
\end{remark}

Let us make precise the notion of $(X_\infty,\omega_\infty)$ as 
a ``limit'' of $(X_g,\omega_g)$. We establish canonical 
piecewise-affine embeddings $\iota_g : K_g \to X_g$, where the 
$K_g$ are the subsurfaces defined in the proof of 
Theorem~\ref{T:infty}, in such a way that 
$\iota_g^*\, |\omega_g|$ converges to $|\omega_\infty|$ on compact 
subsets of $X_\infty$ as $g \to \infty$. (Here $|\omega_n|$ 
indicates the metric induced on $X_n$ by $\omega_n$, 
$3 \le n \le \infty$.) In fact, each $\iota_g$ will be defined on 
an open set $U_g$ containing $K_g$ and dense in $X_\infty$.

For each $3 \le g < \infty$, let $U_g$ be the surface obtained 
from Figure~\ref{F:infty} by making all identifications up through 
index $\lfloor g/2 \rfloor$ for the $A_i$s and $B_i$s, and all 
identifications up through index $\lfloor (g-1)/2 \rfloor$ for the 
$C'_i$s and the $C''_i$s. (Here and in what follows 
$x \mapsto \lfloor x \rfloor$ denotes the ``floor'' function.) 
Retract the union of the triangles 
\begin{gather*}
\left\{
	\left(\frac{1}{2},\frac{1}{2}\right),
	\left(1,
	\frac{2^{\lfloor g/2 \rfloor}-1}%
	      {2^{\lfloor g/2 \rfloor}}
	\right),
	(1,1)
\right\} 
\quad\text{and}\quad
\left\{
	\left(\frac{1}{2},\frac{1}{2}\right),(1,1),
	\left(
	\frac{2^{\lfloor (g-1)/2 \rfloor}-1}%
	      {2^{\lfloor (g-1)/2 \rfloor}},
	1\right)
\right\}
\end{gather*}
onto the triangle 
$\{(1/2,1/2),(1,1-1/2^{\lfloor g/2 \rfloor}),
(1-1/2^{\lfloor (g-1)/2 \rfloor},1)\}$ by a homeomorphism, affine 
on each of the original triangles. Now a surface of genus $g$ with 
two punctures can be created directly by identifying the ``free'' 
edge of this triangle with one of the ``free'' segments on the 
leftmost edges of the polygon. ({\sc n.b.}: at this stage, this 
final identification is not by a translation, but it can be 
chosen to be affine.)\smallskip

\begin{figure}[h]
\centering

\includegraphics[scale=.85]{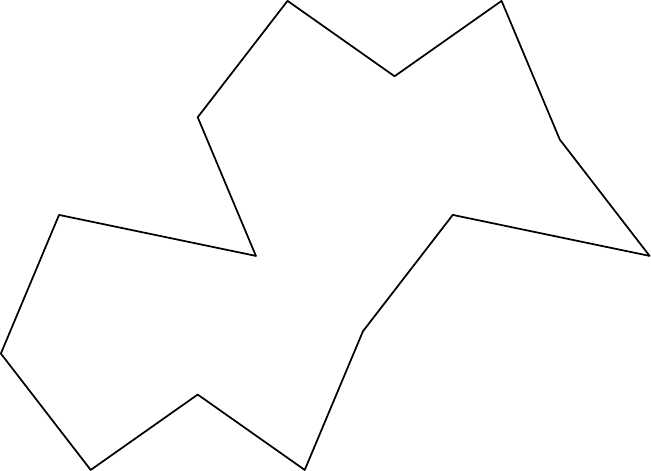} \hspace{1.5cm}
\includegraphics[scale=.85]{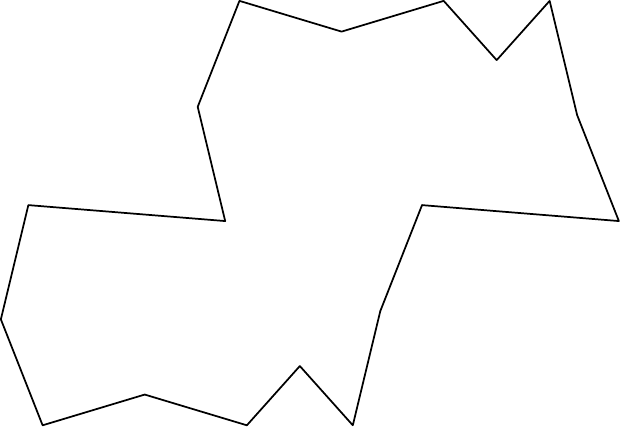} \vspace{1cm}

\includegraphics[scale=.85]{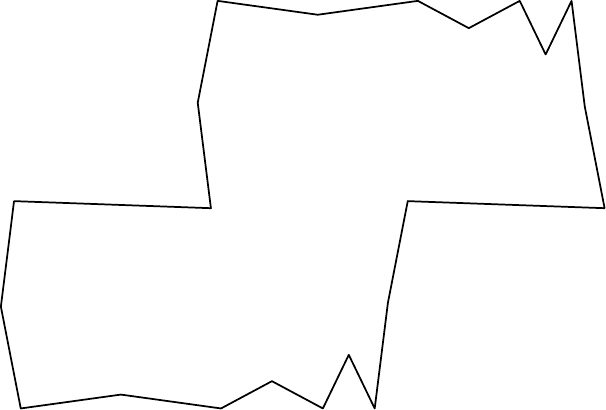} \hspace{1.8cm}
\includegraphics[scale=.85]{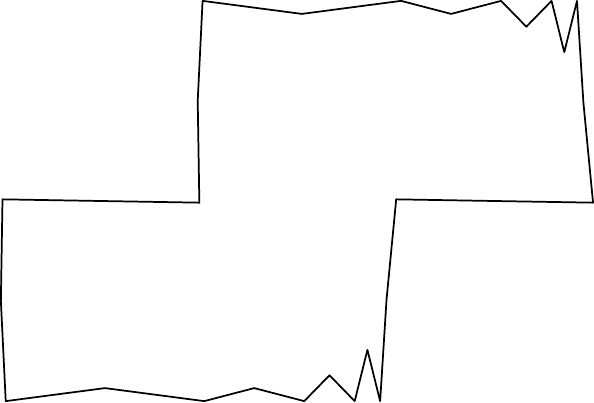} \smallskip

\caption{Outlines of the surfaces $(X_g,\omega_g)$ for 
$g = 3,4,5,6$.}\label{F:seq}
\end{figure}

Figure~\ref{F:seq} shows the outlines of the first few surfaces in 
the sequence $(X_g,\omega_g)$. By adjusting the positions of the 
triangles in the upper right and upper left corners (e.g., 
removing the triangles labelled $T_{2g-\lfloor g/2\rfloor}$ 
through $T_{2g}$, in addition to their mirror images, and regluing 
them along their longest edges in the appropriate location---cf.\ 
Figure~\ref{F:tri} and the description in \S\ref{S:step2tri}), 
one finds that there is a piecewise-affine map $\iota_g$ carrying 
$U_g$ to $X_g$. Moreover, because $U_{g-1} \subset U_g$, $\iota_g$ 
restricts to an embedding of $U_{g-1}$, as well.

\begin{theorem}\label{T:limit}
The metrics $\iota_g^*\,|\omega_g|$ converge to $|\omega_\infty|$ 
uniformly on compact subsets of $X_\infty$.
\end{theorem}
\begin{proof}
Any compact $K \subset X_\infty$ is contained in some $U_g$. For 
any pair of points $P',P'' \in K$, the ratio of the distance from 
$P'$ to $P''$ in each of the metrics $\iota_g^* |\omega_g|$ and 
$|\omega_\infty|$ is bounded by the quasi-conformal constants and 
the Jacobian determinants of the maps $\iota_g$, which are 
uniformly bounded over all of $K$. As these constants approach 
$1$, so do the ratios of lengths over $K$, uniformly.
\end{proof}

\section{The affine group of $(X_\infty,\omega_\infty)$}%
\label{S:ietaff}

In this section we will explore some of the geometry and dynamics 
of $(X_\infty,\omega_\infty)$, culminating in a proof of the 
following:

\begin{theorem}\label{T:affinfty}
$\Lie{Aff}(X_\infty,\omega_\infty) \cong \Z \times \Z/2\Z$ is 
generated by $\psi_\infty$ and $\rho_\infty$.
\end{theorem}

\subsection{An exchange on infinitely many intervals}

Let us revise our definition of ``interval exchange map'' to 
include injective maps from an interval to itself that are upper 
semicontinuous piecewise isometries. (This keeps with the 
``continuous at left endpoints'' convention, although we may lose 
the property of bijectivity, as we shall see in a later example.) 
Then the vertical foliation of $(X_\infty,\omega_\infty)$ induces 
an interval exchange map $f_\infty : [1,0) \to [1,0)$, which can 
also be defined by a two-step process: first, swap the two halves 
of each interval 
$[\frac{2^n - 1}{2^n},\frac{2^{n+1} - 1}{2^{n+1}})$, for 
$0 \le n < \infty$, then swap $[1,1/2)$ with $[1/2,1)$; 
cf.~\S\ref{S:iet}.

We can encode $f_\infty$ symbolically as follows: if we do not 
allow the binary expansion of a number to terminate with only 
$1$s, then each number in $[0,1)$ has a unique binary expansion. 
Use these to identify $[0,1)$ with the set 
$\Space{B} \subset (\field{F}_2)^\field{N}$ consisting of 
sequences that do not terminate with only $1$s. Given a sequence 
$a = a_0 a_1 a_2 \cdots$, we obtain $f_\infty(a)$ as follows:
\begin{enumerate}
\item find the first $i \in \field{N}$ such that $a_i = 0$, and 
replace $a_{i+1}$ with $a_{i+1}+1$;
\item replace $a_0$ with $a_0 + 1$.
\end{enumerate}
The inverse map $f_\infty\inv$ simply reverses these two steps. 
Both $f_\infty$ and $f_\infty\inv$ are bijections. We remark that 
the first return map of $f_\infty$ on either $[0,1/2)$ or 
$[1/2,1)$ is simply the restriction of $f_\infty^2$ to the 
respective interval.

To aid our study at this point, we use the map $r$ defined in 
\S\ref{S:step2tri} along with the following:
\begin{gather*}
h'(x) = \frac{x}{2}, \qquad\quad
h''(x) = (r \circ h')(x) = \frac{x}{2} + \frac{1}{2}, \\
h_\infty(x) = (h' \circ r)(x) = 
\begin{cases}
\frac{1}{2}\left(x + \frac{1}{2}\right), & x \in [0,\frac{1}{2}) 
\\
\frac{1}{2}\left(x - \frac{1}{2}\right), & x \in [\frac{1}{2},1)
\end{cases} 
\end{gather*}
In terms of binary expansions, we can describe the effects of 
these functions on a sequence $a \in \Space{B}$ as follows:
\begin{itemize}
\item $r$ replaces $a_0$ with $a_0 + 1$;
\item $h'$ appends a $0$ to the beginning of the sequence;
\item $h''$ appends a $1$ to the beginning of the sequence;
\item $h_\infty$ replaces $a_0$ with $a_0 + 1$ and appends a $0$ 
to the beginning of the sequence.
\end{itemize}
The formalism of encoding these maps to act on infinite binary 
sequences makes immediate the following result.

\begin{lemma}\label{L:conj}
Let $f_\infty$, $r$, $h'$, $h''$, and $h_\infty$ act on 
$\Space{B}$ as above. Then:
\begin{itemize}
\item $r$ conjugates $f_\infty$ to $f_\infty\inv$.
\item $h'$ conjugates $f_\infty^2\vert_{[0,\,1/2)}$ to 
$f_\infty\inv$.
\item $h''$ conjugates $f_\infty^2\vert_{[1/2,\,1)}$ to 
$f_\infty$.
\item $h_\infty$ conjugates $f_\infty^2\vert_{[0,\,1/2)}$ to 
$f_\infty$.
\end{itemize}
\end{lemma}

\begin{proof}
We will prove the second claim. It is equivalent to show that 
$f_\infty^2 h' f_\infty = h'$. Let $a = a_0 a_1 a_2 \cdots$ be a 
sequence in $\Space{B}$, and let $i_0 \ge 0$ be the first value 
for which $a_{i_0} = 0$. Then $(h'f_\infty(a))_0 = 0$, 
$(h'f_\infty(a))_1 = a_0+1$, $(h'f_\infty(a))_{i_0+2} = 
a_{i_0+1}+1$, and $(h'f_\infty(a))_{i+1} = a_i$ for all other $i$. 
Applying $f_\infty$ to $h'f_\infty(a)$ results in 
$(1,a_0,\dots,a_{i_0-1},0,a_{i_0+1}+1,a_{i_0+2},\dots)$. Now 
$i_0+1$ is the first index $i$ such that 
$(f_\infty h'f_\infty a)_{i} = 0$. Applying $f_\infty$ again 
replaces $(f_\infty h'f_\infty a)_{i_0+2}$ with $a_{i+1}$ and 
changes the leading $1$ to a $0$, so that 
$f_\infty^2 h' f_\infty(a) = h'(a)$.

The proofs of the other claims are similar; in fact, the first 
claim is trivial, while the latter two claims follow from the 
first two.
\end{proof}

\begin{remark}
As a caveat regarding exchanges of infinitely many intervals, we 
describe the interval exchange $F_\infty$ induced on a vertical 
segment by the horizontal foliation. We use the horizontal flow in 
the positive $x$-direction, in which case $F_\infty$ has the 
following effect on $\Space{B}$: for each sequence $a$, 
\begin{enumerate}
\item find the least $i > 0$ such that $a_i \ne a_0$;
\item replace $a_{i-1-2j}$ with $a_{i-1-2j} + 1$ for all 
$0 \le j \le \lfloor i/2 \rfloor$.
\end{enumerate}
Note that this algorithm fails to define $F_\infty$ on the zero 
sequence $\bar{0}$; we will see momentarily that $1/3$ does not 
have a preimage by $F_\infty$, and so we can define $F_\infty(0) 
= 1/3$ without compromising the injectivity or semicontinuity of 
$F_\infty$. The inverse $F_\infty\inv$ acts on $\Space{B}$ as 
follows: for each sequence $a$, 
\begin{enumerate}
\item find the least $i > 0$ such that $a_i = a_{i-1}$;
\item replace $a_{i-1-2j}$ with $a_{i-1-2j} + 1$ for all 
$0 \le j \le \lfloor i/2 \rfloor$.
\end{enumerate}
This algorithm fails for \emph{two} points in $\Space{B}$, namely 
$\overline{01} = 1/3$ and $\overline{10} = 2/3$; these have no 
pre-images by $F_\infty$. Hence we can ``fix'' $F_\infty$ by 
defining $F_\infty(0)$ to be either $1/3$ or $2/3$, but the choice 
is arbitrary. In either case, $F_\infty$ will still not have all 
of $\Space{B}$ as its image. The special role of $1/3$ and $2/3$ 
will be useful to keep in mind (see the proof of 
Lemma~\ref{L:vertnotequiv}).
\end{remark}

Let $\Space{D} \subset [0,1)$ denote the set of dyadic rationals 
in $[0,1)$---that is, the set of rational numbers of the form 
$n/2^m$ for some $n,m \in \Z$. $\Space{D}$ sits inside $\Space{B}$ 
as the set of sequences that are eventually $0$. For each 
$x \in [0,1)$, let $\mathcal{O}^\pm(x)$ be the orbit of $x$ under 
$f_\infty^{\pm1}$. 

\begin{lemma}\label{L:finforbit}
$\Space{D} = \mathcal{O}^\pm(0) \sqcup \mathcal{O}^\pm(1/2)$.
\end{lemma}

A more complete way to state this result is that the union of the 
forward and backward orbits of a sequence $a \in \Space{D}$ is 
entirely determined by the parity of the number of $1$s in the 
sequence $a$. We call $\mathrm{TM}(a) = \sum a_i \in \field{F}_2$ 
the \dfn{Thue--Morse function}: for any particular 
$a \in \Space{D}$, this sum has finitely many terms, and 
$\mathrm{TM}(a)$ is invariant under $f_\infty$ because two digits 
are changed from $a$ to $f_\infty(a)$. We also define the 
\dfn{index of $a$} to be the smallest natural number 
$\mathrm{Ind}(a) \in \field{N}$ such that $a_i = 0$ for all 
$i > \mathrm{Ind}(a)$. (Recall that our sequences in $\Space{B}$ 
start with $a_0$, and so $\mathrm{Ind}(\bar{0}) = 
\mathrm{Ind}(1\bar{0}) = 0$.) We will show that the following 
table determines which orbit contains 
$a \in \Space{D} - \{\bar{0},1\bar{0}\}$:
\begin{equation}\label{Eq:table}
\begin{tabular}{l c | c c }
& & \multicolumn{2}{c}{$\mathrm{TM}(a)$} \\
& & $0$ & $1$ \\ \hline
& even & $\mathcal{O}^-(0)$ & $\mathcal{O}^+(1/2)$ \\
\raisebox{1.3ex}[0pt]{$\mathrm{Ind}(a)$\quad}& 
odd & $\mathcal{O}^+(0)$ & $\mathcal{O}^-(1/2)$ 
\end{tabular}
\end{equation}
One consequence of the proof will be a quick algorithm for 
computing the exact value of $n \in \Z$ such that 
$f_\infty^n(0) = a$ or $f_\infty^n(1/2) = a$.

\begin{proof}[Proof of Lemma~\ref{L:finforbit}]
Let $H$ be the semigroup of functions $\Space{B} \to \Space{B}$ 
consisting of words in $h'$ and $h''$. The map from $H$ to 
$\Space{B}$ defined by $w \mapsto w(\bar{0})$ induces a 
set-theoretic bijection between $\Space{D}$ and the quotient of 
$H$ by the relation $w \sim w h'$. Throughout the proof, we will 
use the equivalence $\Space{D} \leftrightarrow H/\!\sim$, by which 
$(a_0, a_1, \dots, a_{\mathrm{Ind}(a)}, 0, \dots)$ corresponds to 
the equivalence class of 
$\eta_0 \eta_1 \cdots \eta_{\mathrm{Ind}(a)}$, with 
\[
\eta_i = \begin{cases}
h' & \text{if}\ a_i = 0 \\
h'' & \text{if}\ a_i = 1
\end{cases}.
\]
In particular, $\eta_{\mathrm{Ind(a)}} = h''$ if 
$\mathrm{Ind}(a) \ge 1$.

Let $a \in \Space{D}$. We proceed by induction on 
$\mathrm{Ind}(a)$. Direct computation shows that 
\[
h'' h''(\bar{0}) = f_\infty h'(\bar{0}) 
= f_\infty(\bar{0})
\qquad\text{and}\qquad
h'h''(\bar{0}) = f_\infty\inv h''(\bar{0}) 
= f_\infty\inv(1\bar{0}),
\]
and therefore if $\mathrm{Ind}(a) = 1$, $a$ is in the union 
of the orbits of $\bar{0}$ and $1\bar{0}$. Now suppose 
$\mathrm{Ind}(a) \ge 2$, and let 
$w = \eta_0 \eta_1 \cdots \eta_{\mathrm{Ind}(a)-1} h''$ be the 
corresponding word in $H$. Using the above computations, we can 
rewrite the effect of $w$ on $\bar{0}$ in the following way:
\[
w(\bar{0}) = \begin{cases}
\eta_0 \eta_1 \cdots \eta_{\mathrm{Ind}(a)-2} 
  f_\infty h'(\bar{0}) 
& \text{if}\ \eta_{\mathrm{Ind}(a)-1} = h'' \\
\eta_0 \eta_1 \cdots \eta_{\mathrm{Ind}(a)-2} 
  f_\infty\inv h''(\bar{0}) 
& \text{if}\ \eta_{\mathrm{Ind}(a)-1} = h'
\end{cases}.
\]
From Lemma~\ref{L:conj}, we have 
\[
f_\infty^2 h' = h' f_\infty\inv 
\qquad\text{and}\qquad
f_\infty^2 h'' = h'' f_\infty.
\]
These relations allow us to move $f_\infty$ to the far left of the 
word, at each step exchanging a power of $f_\infty$ for a power 
whose absolute value is twice as great, which means we have 
expressed $a$ as $f_\infty^n(b)$, where $\mathrm{Ind}(b) < 
\mathrm{Ind}(a)$. Here $|n| = 2^{\mathrm{Ind}(a)-1}$, and the 
sign of $n$ is determined by the number of $0$s among 
$a_0,\dots,a_{\mathrm{Ind}(a)-1}$. By induction, we have shown 
that every point of $\Space{D}$ lies in the union of the orbits 
of $\bar{0}$ and $1\bar{0}$. 

Because $\mathrm{TM}(a)$ is invariant under $f_\infty$, $\bar{0}$ 
and $1\bar{0}$ are not in the same orbit, and therefore 
$\Space{D}$ is a disjoint union of these two orbits.
\end{proof}

\begin{remark}
It is not hard to show that both $f_\infty$ and $F_\infty$ are 
ergodic, for example, using elementary linear algebra. It is less 
clear how the flow in other directions on 
$(X_\infty,\omega_\infty)$ behaves.
\end{remark}

\begin{remark}\label{R:avatardist}
We will need a bit more information about the points of 
discontinuity of $f_\infty$. These correspond precisely to 
sequences of the form $11\cdots11\bar{0}$ or $11\cdots1101\bar{0}$ 
(the initial number of $1$s may be zero). From the information in 
\eqref{Eq:table}, we see that the forward and backward orbits of 
both $0$ and $1/2$ each contain infinitely many such points.
\end{remark}

\subsection{Vertical trajectories and the Veech group of 
$(X_\infty,\omega_\infty)$}

\begin{lemma}\label{L:scdense}
Saddle connections are dense in the vertical foliation of 
$(X_\infty,\omega_\infty)$. Every vertical critical trajectory is 
a saddle connection. 
\end{lemma}
\begin{proof}
Let $x \in \Space{D}$, and consider the point $(x,0)$ on the 
boundary of the unit square. If $x$ is not already a point of 
discontinuity of $f_\infty$, then by Lemma~\ref{L:finforbit} and 
Remark~\ref{R:avatardist}, there exist $m,n > 0$ such that 
$f_\infty^{-m}(x)$ and $f_\infty^n(x)$ are points of discontinuity 
of $f_\infty$. Because $f_\infty$ is determined by the vertical 
flow, this means there is a vertical saddle connection passing 
through $(x,0)$ and connecting $(f_\infty^{-m}(x),0)$ to 
$(f_\infty^n(x),0)$. If $x$ is a point of discontinuity of 
$f_\infty$, then so is $f_\infty(x)$, and there is a vertical 
saddle connection from $(x,0)$ to $(f_\infty(x),1)$.
\end{proof}

The proof shows, moreover, that the union of the vertical critical 
trajectories contains precisely those points that have 
representatives in Figure~\ref{F:infty} with a dyadic rational 
$x$-coordinate. 

For clarity in the proof of Lemma~\ref{L:vertnotequiv}, 
we state the following definition and proposition.

\begin{definition}
An ({\em open}) {\em angular sector} is a Riemannian surface 
isometric to the half-infinite strip 
\[
U_{t,\Theta} = 
\{ z = x + iy \mid x < \log t,\ 0 < y < \Theta \}
\]
with the (conformal) metric $|e^z\,\D{z}|$. $\Theta$ is called 
the {\em angle} of the sector, and $t$ is its {\em radius}.
\end{definition}

\begin{proposition}\label{P:sector}
Let $(X,\omega)$ be a translation surface, and let 
$\phi \in \Lie{Aff}(X,\omega)$ be an affine homeomorphism. 
Suppose $X$ contains an embedded angular sector $U$ whose 
angle is an integer multiple of $\pi$. Then $\phi(U)$ contains 
an angular sector with the same angle as $U$.
\end{proposition}

To see why this proposition is true, it suffices to consider 
the case of a sector with angle $\pi$. For then $\phi$ transforms 
the sector into a half-ellipse, which thus contains a sector 
with angle $\pi$ and the same center as the ellipse.

\begin{lemma}\label{L:vertnotequiv}
The vertical direction of $(X_\infty,\omega_\infty)$ is not 
affinely equivalent to any other direction on 
$(X_\infty,\omega_\infty)$.
\end{lemma}

\begin{proof}
Let $\mathcal{F}_v$ be the vertical foliation of 
$(X_\infty,\omega_\infty)$, and let $\mathcal{F}_\theta$ be the 
foliation in some other direction $\theta$. Assume there exists 
some $\phi \in \Lie{Aff}(X_\infty,\omega_\infty)$ that sends 
$\theta$ to the vertical direction. Let $L$ be the critical leaf 
of $\mathcal{F}_\theta$ emanating from $(0,2/3)$ in 
Figure~\ref{F:infty}. Then $\phi(L)$ must be a critical trajectory 
in the vertical direction, which means it must be a saddle 
connection, by Lemma~\ref{L:scdense}. By composing $\phi$ with 
some power of $\psi_\infty$ and $\rho_\infty$, if necessary, we 
may assume $\phi(L)$ is the saddle connection $L_0$ from $(0,0)$ 
to $(0,1/2)$. 

Now consider an angular sector 
$U = \mathrm{image}\, (U_{\eps,2\pi} \to X_\infty)$, with 
$\eps < 1/8$, such that the radius in the direction of angle $\pi$ 
is sent to a portion of $L_0$. By Proposition~\ref{P:sector}, 
because the angle of $U$ is an integer multiple of $\pi$, 
$\phi\inv(U)$ must also be a sector of angle $2\pi$. But this 
is impossible, because the two halves of $U$ on either side of 
$L_0$ must be sent to sectors of radius $\pi$ with $\phi\inv(L_0) 
= L$ as a boundary radius; no such sectors exist, due to the 
accumulation of saddle connections at $(0,2/3)$. 
Therefore no affine homeomorphism can send the vertical direction 
of $(X_\infty,\omega_\infty)$ to any other direction.
\end{proof}

\begin{remark}
The distinction between vertical critical trajectories on 
$(X_\infty,\omega_\infty)$ and those emanating from the points 
$(1/2,1/3)$ and $(0,2/3)$ in Figure~\ref{F:infty} can also be 
made using a purely topological criterion, rather than the 
geometric criterion of Proposition~\ref{P:sector}. This amounts 
to a study of the space of critical trajectories on a translation 
surface of infinite type, which is the subject of \cite{jpBfV10}.
\end{remark}

Now we are ready to prove the main theorem of this section.

\begin{proof}[Proof of Theorem~\ref{T:affinfty}]
By Lemma~\ref{L:vertnotequiv}, any affine homeomorphism $\phi$ of 
$(X_\infty,\omega_\infty)$ must preserve the vertical direction. 
Because it must preserve the set of saddle connections, and the 
lengths of the vertical saddle connections are all powers of $2$, 
the derivative of $\phi$ must act on the vertical direction by 
multiplication by $\pm 2^n$ for some $n \in \Z$. By composing 
$\phi$ with a power of $\psi_\infty$ and $\rho_\infty$, if 
necessary, we may assume that $\phi$ is orientation-preserving 
and the derivative of $\phi$ is the identity in the vertical 
direction. Note that, because the area of 
$(X_\infty,\omega_\infty)$ is finite, the derivative of $\phi$ 
must lie in $\Lie{SL}_2(\R)$, which implies that its only 
eigenvalue is $1$.

Thus $\phi$ is either a translation automorphism or a parabolic 
map. The latter is impossible because $(X_\infty,\omega_\infty)$ 
does not have any cylinders in the vertical direction. The 
existence of non-trivial translation automorphisms is ruled out 
directly, for example by observing that each vertical saddle 
connection has only one other of the same length (its image by 
$\rho_\infty$), and no translation automorphism can carry one to 
the other. Therefore the original map $\phi$ was a product of a 
power of $\psi_\infty$ and $\rho_\infty$, and the result is 
proved.
\end{proof}

\begin{remark}\label{R:veech}
In the proof of Theorem~\ref{T:affinfty}, we showed that 
$(X_\infty,\omega_\infty)$ has no non-trivial translation 
automorphisms. The same is true of the finite-genus surfaces 
$(X_g,\omega_g)$: as observed in the proof of 
Corollary~\ref{C:nothyp}, each $(X_g,\omega_g)$ (with $g \ge 4$) 
has a saddle connection such that no other saddle connection has 
the same developing vector; this rules out the possibility of 
$\Lie{Aff}(X_g,\omega_g)$ containing a non-trivial translation. 
A similar argument works for $g = 3$. 
We conclude that for any $3 \le g \le \infty$, every affine map 
of $(X_g,\omega_g)$ is uniquely determined by its derivative; 
this means that $\Lie{Aff}(X_g,\omega_g)$ can be identified with 
the Veech group $\Gamma(X_g,\omega_g)$ in $\Lie{GL}_2(\R)$. We 
can thus compare the groups $\Lie{Aff}(X_g,\omega_g)$ as 
subgroups of $\Lie{GL}_2(\R)$, even though {\em a priori} they 
are groups of homeomorphisms of surfaces with different genera.

Recall that a sequence $\{\Gamma_n\}_{n=0}^\infty$ of closed 
subgroups of $\Lie{GL}_2(\R)$ {\em converges geometrically} 
to $\Gamma_\infty \subset \Lie{GL}_2(\R)$ if both of the 
following hold:
\begin{enumerate}
\item If $\{\gamma_n \in \Gamma_n\}$ is a sequence of elements 
converging to $\lim \gamma_n = \gamma_\infty$, then 
$\gamma_\infty \in \Gamma_\infty$.
\item Any $\gamma \in \Gamma_\infty$ is obtained as a limit of 
$\gamma_n \in \Gamma_n$ as in (1).
\end{enumerate}
It follows immediately from the definitions and from 
Lemma~\ref{L:1/2} that the sequence $\gen{\psi_g,\rho_g} \subset 
\Lie{Aff}(X_g,\omega_g)$ converges geometrically to 
$\gen{\psi_\infty,\rho_\infty} = 
\Lie{Aff}(X_\infty,\omega_\infty)$. A natural question is whether 
the groups $\Lie{Aff}(X_g,\omega_g)$ with $g$ finite converge 
geometrically to $\Lie{Aff}(X_\infty,\omega_\infty)$. This 
question is of particular interest since it is currently unknown 
whether there exists a translation surface of finite genus whose 
affine group contains a finite-index cyclic subgroup generated by 
a pseudo-Anosov element. If it is true that 
$\Lie{Aff}(X_g,\omega_g)$ converges to 
$\Lie{Aff}(X_\infty,\omega_\infty)$, then this would at least 
show that the Veech groups of $(X_g,\omega_g)$ for $g$ large 
enough are ``close to'' cyclic groups.

This result is not yet known, although early investigations with 
other families of surfaces that converge (uniformly on compact 
subsets) to a limit surface suggest that one can in general 
expect the geometric limit of the Veech groups to be contained 
in the Veech group of the limiting surface. In the case of the 
Arnoux--Yoccoz surfaces, since we have subgroups of 
$\Lie{Aff}(X_g,\omega_g)$ converging to 
$\Lie{Aff}(X_\infty,\omega_\infty)$, we would then in fact have 
the equality $\lim_{n\to\infty} \Lie{Aff}(X_g,\omega_g) = 
\Lie{Aff}(X_\infty,\omega_\infty)$, taking the geometric limit.
\end{remark}

\section*[Appendix]{Appendix. From the top: $g = 1,2$}

In \S\ref{S:limit}, we extended the family of Arnoux--Yoccoz 
surfaces $(X_g,\omega_g)$ to the index $g = \infty$. In this 
appendix we extend the construction of \S\ref{S:step2tri} to 
create $(X_1,\omega_1)$ and $(X_2,\omega_2)$ so that the 
sequence $(X_g,\omega_g)$ is defined for all indices 
$1 \le g \le \infty$.

\subsection*{$g = 1$}

The defining equation for $\alpha$ in this case is $\alpha = 1$. 
The corresponding surface is a torus, formed from the unit square 
by the usual top-bottom and left-right identifications. Hence 
$(X_1,\omega_1) = (\C/(\Z\oplus{i}\Z),\D{z})$ and $\psi_1$ is the 
identity map.

\subsection*{$g = 2$}

Here we get the defining equation $\alpha^2 + \alpha = 1$, which 
means that $\alpha = (\sqrt{5}-1)/2$ is the inverse of the golden 
ratio, as mentioned in the introduction. Beginning with the unit 
square, a single square of side length $1 - \alpha = \alpha^2$ is 
removed from the upper right corner. Two slits are made, one from 
$(\alpha/2,0)$ to $(\alpha/2,1)$ and the other from 
$((1+\alpha)/2,0)$ to $((1+\alpha)/2,\alpha)$, thereby cutting 
the staircase into three separate pieces. After the usual 
identifications are made, following the procedure of 
\S\ref{S:step2tri}, the result is a disconnected pair of tori. 
This is to be expected: the corresponding interval exchange map 
$f_2$ is reducible. Viewed on the circle $[0,1]/\{0\sim1\}$, it 
splits into two interval exchanges, each of which swaps a pair of 
segments whose lengths are in the golden ratio. The pair of tori 
taken together admits a pseudo-Anosov homeomorphism $\psi_2$ with 
expansion constant $1/\alpha = (1+\sqrt{5})/2$, which in the 
process exchanges the components.

Genus $2$ is not entirely absent in this picture, however. 
Indeed, the surface constructed above is a limit of surfaces 
in $\mathcal{H}(1,1)$ and therefore lies in the principal 
boundary of this stratum. If we shorten the height of the first 
slit in the paragraph above to $1-\eps$ and that of the second 
slit to $\alpha-\eps$, then the same identifications are 
possible, and we obtain a connected sum of the two tori, 
resulting in two cone points of angle $4\pi$. As $\eps \to 0$, 
the two cone points collapse into a single point, which becomes 
a marked point on each of the two tori.

Moreover, the period lattices of the two tori that compose 
$X_2$ satisfy a remarkable property: if either is scaled by 
a factor of $\sqrt{5}$, the result is a sublattice of index 
$5$ in the other. This implies that $(X_2,\omega_2)$ lies in 
the boundary of the ``eigenform locus'' defined by McMullen 
\cite{ctM07} (see also \cite[\S6]{kC04}), which is composed of 
surfaces $(X,\omega)$ in $\mathcal{H}(1,1)$ such that the 
Jacobian variety of $X$ admits real multiplication with $\omega$ 
as an eigenform. (The author thanks Barak Weiss for pointing out 
this feature of $(X_2,\omega_2)$.)

Because $(X_2,\omega_2)$ is not connected, we adopt the convention 
that the group $\Lie{Aff}(X_2,\omega_2)$ only consists of affine 
self-maps each of which has constant derivative. The 
orientation-reversing map $\rho_2 \in \Lie{Aff}(X_2,\omega_2)$ 
exchanges the components. By composing any 
$\phi \in \Lie{Aff}(X_2,\omega_2)$ with $\rho_2$ or $\psi_2$, if 
necessary, we may assume that $\phi$ is orientation-preserving 
and also preserves the components of $X_2$. The 
orientation-preserving affine group of a torus with a marked point 
is $\Lie{SL}_2(\Z)$; as was the case in Remark~\ref{R:veech}, the 
derivative homomorphism is an isomorphism. Thus, to compute the 
remainder of $\Lie{Aff}(X_2,\omega_2)$, we wish to find the 
intersection of the affine groups of the two components. Set 
\[
M_1 = \begin{pmatrix} 1 & -\alpha \\ \alpha & 1 \end{pmatrix}
\qquad\text{and}\qquad
M_2 = \begin{pmatrix} \alpha & -1 \\ 1 & \alpha \end{pmatrix}.
\]
Following a certain normalization, the two components of $X_2$ 
have the columns of $M_1$ and $M_2$ for their respective homology 
bases. Then we want to determine 
\[
(M_1 \cdot \Lie{SL}_2(\Z) \cdot M_1\inv) \cap 
	(M_2 \cdot \Lie{SL}_2(\Z) \cdot M_2\inv)
\]
or, equivalently, 
$(M_2\inv M_1 \cdot \Lie{SL}_2(\Z) \cdot M_1\inv M_2) 
\cap \Lie{SL}_2(\Z)$. We have 
\[
M_1\inv M_2 = (M_2\inv M_1)^\top = 
	\frac{\alpha}{2-\alpha} 
	\begin{pmatrix}
	2 & -1 \\ 1 & 2
	\end{pmatrix}
\]
and so we want to find the quadruples of integers $(X,Y,Z,W)$ with 
$XW - YZ = 1$ such that the following is in $\Lie{SL}_2(\Z)$:
\[
M_2\inv M_1 
  \begin{pmatrix} X & Y \\ Z & W \end{pmatrix} 
  M_1\inv M_2 
= \frac{1}{5} 
	\begin{pmatrix}
	4X + 2(Y + Z) + W & 4Y + 2(W - X) - Z \\
	4Z + 2(W - X) - Y & 4W - 2(Y + Z) + X
	\end{pmatrix}.
\]
That is, each of the entries in the final matrix must be 
congruent to $0$ modulo $5$. This is a necessary and sufficient 
condition. All four entries yield the same linear condition 
$X + 3Y + 3Z + 4W \equiv 0$ mod $5$, which is satisfied in 
particular if $X \equiv W \equiv 1$ and $Y \equiv Z \equiv 0$ 
mod $5$. Hence the Veech group of $(X_2,\omega_2)$ contains a copy 
of the principle $5$-congruence subgroup of $\Lie{SL}_2(\Z)$; 
it is therefore a lattice in $\Lie{SL}_2(\R)$.


\bibliographystyle{math}
\bibliography{ay}

\end{document}